\documentclass[11pt,reqno]{amsart}
\usepackage{wasysym,amssymb,eufrak,indentfirst,graphicx,cite,amsthm,color}
\usepackage[bookmarksnumbered, colorlinks, linktocpage, plainpages]{hyperref}
\usepackage{epsfig}
\usepackage{float}
\usepackage{indentfirst}
 \newtheorem{thm}{Theorem}
 \newtheorem{cor}{Corollary}
 \newtheorem{lem}{Lemma}
 \newtheorem{prop}{Proposition}
 \theoremstyle{definition}
 \newtheorem{defn}{Definition}
 \theoremstyle{remark}
 \newtheorem{rem}{Remark}
 \newtheorem{ex}{Example}
\begin{document}

\title[Julia theory for slice regular functions]
{Julia theory for slice regular functions}

\author[G. B. Ren]{Guangbin Ren}
\address{Guangbin Ren, Department of Mathematics, University of Science and
Technology of China, Hefei 230026, China}

\email{rengb$\symbol{64}$ustc.edu.cn}

\author[X. P. Wang]{Xieping Wang}
\address{Xieping Wang, Department of Mathematics, University of Science and
Technology of China, Hefei 230026,China}

\email{pwx$\symbol{64}$mail.ustc.edu.cn}

\thanks{This work was supported by the NNSF  of China (11371337), RFDP (20123402110068).}


\begin{abstract}
Slice regular functions have been extensively studied over the past decade, but much less is known about their boundary behavior. In this paper, we initiate the study of  Julia theory for slice regular functions. More specifically, we establish the quaternionic versions of the Julia lemma, the Julia-Carath\'{e}odory theorem, the boundary Schwarz lemma, and the Burns-Krantz rigidity theorem  for slice regular self-mappings of the open unit ball $\mathbb B$ and of the right half-space $\mathbb H^+$. Our quaternionic boundary Schwarz lemma  involves a Lie bracket reflecting  the non-commutativity  of quaternions. Together with some explicit examples, it shows that the slice derivative of a slice regular self-mapping of  $\mathbb B$ at a boundary fixed point is not necessarily a positive real number, in contrast to
that in the complex case, meaning that its commonly believed version turns out to be totally wrong.

\end{abstract}

\maketitle

\section{Introduction}
The celebrated Julia lemma \cite{Julia}  and  the Julia-Carath\'{e}odory theorem \cite{Caratheodory1, Caratheodory2} for holomorphic self-mappings of the open unit disc $\mathbb D\subset \mathbb C$ and of the right half complex plane $\mathbb C^{+}=\{z\in \mathbb C:{\rm{Re}}(z)>0\}$  play an important role in the theory of hyperbolic geometry, complex dynamical systems, and composition operators; see  e.g. \cite{ Abate0, Cowen, Shapiro, Shoikhet}.
These two theorems together with the boundary Schwarz lemma \cite{Herzig} are powerful tools in the theory of iterating holomorphic self-mappings, their fixed points and boundary behaviors; see e.g. \cite{Shoikhet, Cowen}.
There are two canonical approaches from different points of view to the study of  Julia-Carath\'{e}odory theorem.
The usual one is the function-theoretic approach which has a strongly geometric character and depends ultimately on the Schwarz lemma and involves an asymptotic version of the Schwarz lemma, known as the Julia  lemma.
 Sarason initiated the study of the Julia-Carath\'{e}odory theorem via a Hilbert space approach, which puts insight from a different perspective. In that treatment, the Julia lemma emerges as a consequence of the Cauchy-Schwarz inequality; see \cite{Sarason1,Sarason2} for more details.

There are many extensions for these results
 to higher dimensions for holomorphic mappings of several complex variables.
The Julia-Carath\'{e}odory theorem for holomorphic self-mappings on the open unit ball $\mathbb B_n\subset \mathbb C^n$ were studied by  Herv\'{e}
\cite{Herve}  and  by Rudin \cite{Rudin}, and for holomorphic self-mappings on strongly (pseudo)convex domains and other domains in $\mathbb C^n$ by other authors, notably by Abate; see
\cite{Abate0} and the references therein on this subject; see also  \cite{Agler} for the bidisk version and \cite{Abate} for the polydisk version. Recently  a variant of the Julia-Carath\'{e}odory theorem for infinitesimal generators has been investigated  in \cite{Bracc}.
 However,  there are no analogous results, as far as we know,  for other classes of functions, such as the classical regular functions in the sense of Cauchy-Fueter and the recently introduced slice regular functions. A great challenge  arising from  extensions to the  setting of quaternions is the lack of commutativity.

The purpose of the present article is to generalize  the Julia lemma and the Julia-Carath\'{e}odory theorem  as well as  the boundary Schwarz lemma to the setting of quaternions for slice regular functions of one quaternionic variable.

The theory of slice regular functions is initiated recently by Gentili and Struppa \cite{GS1,GS2}. It is significantly different from that of regular functions in the sense of Cauchy-Fueter and has  elegant  applications to the functional calculus for noncommutative operators \cite{Co2} and to Schur analysis \cite{Alpay2}. For the detailed up-to-date theory, we refer to the monographs \cite{GSS, Co2}. The theory of slice regular functions  centers around the non-elliptic  differential operator with nonconstant coefficients, given by
 $$\big|{\rm{Im}}(q)\big|^2\frac{\partial }{\partial x_0}+{\rm{Im}}(q)
 \sum_{j=1}^3 x_j \frac{\partial }{\partial x_j},
$$
where ${\rm{Im}}(q)$ is the imaginary part of the quaternion $q=x_0+{\rm{Im}}(q)\in \mathbb H$; see \cite{CGCS} for more details.
 Furthermore, the notion of slice regularity was also extended to functions of an octonionic variable  \cite{GS50} and to the setting of Clifford algebras  \cite{Co6,Co3} as well as  to the setting of  alternative real algebras \cite{Ghiloni1}.

The study of a geometric theory for slice regular functions of one quaternionic variable has by now produced several interesting results, sometimes analogous to those valid for holomorphic functions. The Bohr theorem \cite{DGS2} is among these results, together with the Bloch-Landau theorem \cite{DGS1} and the Landau-Toeplitz theorem \cite{DS}.
 Recently,  the authors established the growth  and distortion theorems for slice regular extensions of normalized univalent holomorphic functions with the tool of a so-called convex combination identity \cite{RW2}, and set up  the Borel-Carath\'{e}odory theorems for slice regular functions using the method of finite average \cite{RW1}. In the present article, we continue to investigate the geometric properties of slice regular functions and our starting point is the counterpart of Schwarz-Pick theorem in the setting of quaternions, which is first established in \cite{BS}; see \cite{Alpay1} for an alternative and shorter proof by means of the Nevanlinna-Pick interpolation problem.

Our main results in this article  are the  quaternionic versions of  the Julia lemma, the Julia-Carath\'{e}odory theorem,  the Burns-Krantz rigidity theorem as well as the boundary Schwarz lemma for  slice regular  self-mappings of the open unit ball $\mathbb B\subset \mathbb H$ and of the right half-space $\mathbb H^+$. Although some results of the present paper coincide in form with those in the complex setting, they can not be obtained directly from the original complex results using several properties of slice regular functions, except Theorem \ref{B-krantz} and Corollary \ref{218}.

Before presenting our main results, we first recall some notations. Let $\mathbb B$ denote the open unit ball in the quaternions $\mathbb H$. For any $k>0$ and any point $p\in \partial\mathbb B$, we denote
$$\mathcal{S}(p,k)=\big\{q\in \mathbb B:|p-q|^2<k(1-|q|^2)\big\}.$$
It is known that $\mathcal{S}(p,k)$ is an open ball internally tangent to the unit sphere $\partial\mathbb B$ with  center $\frac{p}{1+k}$ and radius $\frac{k}{1+k}$. The boundary  sphere of this ball is called an \textit{orisphere}. These orispheres are crucial in the geometry theory concerning boundary behaviors of slice regular self-mappings of the open unit ball $\mathbb B$ as shown in the following quaternionic counterpart of Julia's lemma.

\begin{thm}\label{Julia}{\bf(Julia)}
Let $f$ be a slice regular self-mapping of the open unit ball $\mathbb B$ and let $\xi\in\partial \mathbb B$. Suppose that there exists a sequence $\{q_n\}_{n\in \mathbb N}\subset \mathbb B$ converging to $\xi$ as $n$ tends to $\infty$, such that the limits
$$\alpha:=\lim\limits_{n\rightarrow\infty}\frac{1-|f(q_n)|}{1-|q_n|}$$
and
$$\eta:=\lim\limits_{n\rightarrow\infty}f(q_n)$$
exist $($finitely$)$. Then $\alpha>0$ and the  inequality
\begin{eqnarray}\label{eq:11}
{\rm{Re}}\Big(\big(1-f(q)\overline{\eta}\big)^{-\ast}\ast\big(1+f(q)\overline{\eta}\big)\Big)
\geq \frac{1}{\alpha}\
{\rm{Re}}\Big(\big(1-q\overline{\xi}\big)^{-\ast}\ast\big(1+q\overline{\xi}\big)\Big)
\end{eqnarray}
holds throughout the open unit ball $\mathbb B$ and is strict except for regular M\"obius transformations of $\mathbb B$.

In particular, the inequality $(\ref{eq:11})$ is equivalent to
\begin{equation}\label{eq:10}
\frac{|\eta-f(q)|^2}{1-|f(q)|^2}\leq \alpha\,\frac{|1-q|^2}{1-|q|^2},
\end{equation}
whenever $\xi=1$.
In other words,
$$f(\mathcal{S}(1,k))\subseteq\mathcal{S}(\eta,\alpha k), \qquad \forall \ k>0.$$
\end{thm}
Inequality $(\ref{Julia})$ is called  Julia's inequality in view of (\ref{eq:10}). It results in the quaternionic version of the Julia-Carath\'{e}odory theorem.

\begin{thm}\label{Julia-Caratheodory}{\bf(Julia-Carath\'{e}odory)}
Let $f$ be a slice regular self-mapping of the open unit ball $\mathbb B$. Then the following conditions are equivalent:
\begin{enumerate}
\item[(i)]  the lower limit
\begin{eqnarray}\label{def:alpha-Julia}\alpha:=\liminf\limits_{q\rightarrow 1}\dfrac{1-|f(q)|}{1-|q|}
\end{eqnarray} is finite,
where the limit is taken as $q$ approaches $1$ unrestrictedly in $\mathbb B$;

\item[(ii)]  there is a number  $\eta\in\partial \mathbb B$ such that the  slice regular quotient $$(q-1)^{-\ast}\ast\big(f(q)-\eta\big)$$ has a non-tangential limit, say $f'(1)$, at the point $1$;
\item[(iii)]  $f$ and its slice derivative $f'$ have  non-tangential limits, say $f(1)$ and $f'(1)$ respectively, at the point $1$.
\end{enumerate}
\bigskip

Moreover, under the above conditions we have
\begin{enumerate}
\item[(a)]  $\alpha>0$ in ${\rm{(i)}}$;
\item[(b)]  the slice derivatives $f'(1)$ in ${\rm{(ii)}}$ and ${\rm{(iii)}}$ are the same;
\item[(c)]  $f'(1)=\alpha \eta$ and $\eta=f(1)$;
\item[(d)]  the quotient $\dfrac{1-|f(q)|}{1-|q|}$ has the non-tangential limit $\alpha$.
\end{enumerate}
\end{thm}

It is worth remarking here that  $f'(1)$ is closely related to $\alpha$ when $\alpha$ is finite. However, when $\alpha=\infty$,  $f'(1)$ can be any quaternion $\beta$ as demonstrated  by the regular polynomial $f(q)={q^n}\beta/n$ with $n>|\beta|$. Incidentally, although Theorem \ref{Julia} and Theorem \ref{Julia-Caratheodory} coincide in form with those in the complex setting, they can not be obtained directly from the original complex results using several properties of slice regular functions.

The quaternionic right half-space version of the Julia-Carath\'{e}odory theorem can also be established and its proof depends ultimately on the right half-space version of the Schwarz-Pick theorem (see Sect. $4$). As a direct consequence, we obtain the Burns-Krantz  rigidity theorem  for slice regular functions with values in the closed right half-space
$$\overline{\mathbb H}^+=\big\{q\in\mathbb H: \textrm{Re}(q)\geq0\big\};$$
see also \cite{Migliorini, GV}.

\begin{thm} \label{B-krantz}{\bf(Burns-Krantz)}
If  $f:\mathbb B\rightarrow \overline{\mathbb H}^{+}$ is  slice regular and $$f(q)=o(|q+1|),  \qquad q\rightarrow -1,$$ then $f\equiv0$.
\end{thm}

The classical boundary Schwarz lemma  claims that
\begin{equation}\label{C-Schwarz}
\overline{\xi} f(\xi)\overline{f'(\xi)}\geq1.
\end{equation}
for any  holomorphic function $f$ on $\mathbb D\cup\{\xi\}$  satisfying
$f(\mathbb D)\subseteq \mathbb D$, $f(0)=0$ and $f(\xi)\in\partial\mathbb D$ for some point $\xi\in\partial\mathbb D$.
However, its quaternionic  variant has an additional  item in terms of Lie brackets reflecting  the non-commutativity of quaternions.

For a given element $\xi\in\mathbb H$, we denote by $[\xi]$ the associated 2-sphere:
$$[\xi]=\big\{q\xi q^{-1}: q\in\mathbb H\setminus\{0\}\big\}.$$
Recall that two quaternions belong to the same sphere if and only if they have the same modulus and the same real part.

\begin{thm}\label{Generalized Herzig}{\bf(Schwarz)}
Let $\xi\in \partial\mathbb B$ and $f$ be a slice regular function on $\mathbb B\cup[\xi]$ such that $f(\mathbb B)\subseteq\mathbb B$ and $f(\xi)\in \partial\mathbb B$. Then:
\begin{enumerate}
\item[(i)]
It holds that
$$\overline{\xi}\Big(f(\xi)\overline{f'(\xi)}+\big[\bar{\xi}, f(\xi)\overline{R_{\bar{\xi}}R_{\xi}f(\xi)}\,\big]\Big)
\geq \frac{2\big(1-|f(0)|\big)^2}{1-|f(0)|^2+|f'(0)|},$$
where $$R_{\xi}f(q):=(q-\xi)^{-\ast}\ast\big(f(q)-f(\xi)\big).$$

\item[(ii)]
If further  $$f^{(k)}(0)=0, \qquad \forall\ k=0,1,\ldots,n-1$$
for some  $n\in\mathbb N$, then
 $$\overline{\xi}\Big(f(\xi)\overline{f'(\xi)}+\big[\bar{\xi}, f(\xi)\overline{R_{\bar{\xi}}R_{\xi}f(\xi)}\,\big]\Big)
 \geq n+\frac{2\big(1-\big|f^{(n)}(0)\big|/n!\big)^2}{1-\big|f^{(n)}(0)/n!\big|^2
 +\big|f^{(n+1)}(0)\big|/(n+1)!}.$$
 In particular, $$\overline{\xi}\Big(f(\xi)\overline{f'(\xi)}+\big[\bar{\xi}, f(\xi)\overline{R_{\bar{\xi}}R_{\xi}f(\xi)}\,\big]\Big)
 > n$$ unless $f(q)=q^nu$ for some $u\in\partial\mathbb B$.
\end{enumerate}
\end{thm}

\begin{cor}\label{Cor-Schwarz}
Let $\xi\in \partial\mathbb B$ and $f$ be a slice regular function on $\mathbb B\cup[\xi]$ such that $f(\mathbb B)\subseteq\mathbb B$, $f(0)=0$ and $f(\xi)=\xi$. Then
$$f'(\xi)-\big[\xi,R_{\bar{\xi}}R_{\xi}f(\xi)\big]>1$$
unless $f(q)=q$ for all $q\in\mathbb B$.
\end{cor}

It is worth remarking here that  the Lie bracket in the preceding corollary does not vanish and $f'(\xi)$ is not necessarily a positive real number, in general; see Example \ref{ex-Schwarz1} in Section 3 for more details.
This means that in the setting of quaternions  the commonly believed fact that
 $$f'(\xi)>1$$ may fail; see \cite[Theorem 9.24]{GSS}. However, the same line of the proof of Theorem \ref{Generalized Herzig} implies simultaneously the following theorem, which of course can also be proven by reducing it, with a few preliminary steps, to the original complex result.

\begin{thm}\label{Schwarz in modulus}
Let $\xi\in \partial\mathbb B$ and $f$ be a slice regular function on $\mathbb B\cup\{\xi\}$ such that $f(\mathbb B)\subseteq\mathbb B$, $f(0)=0$ and $f(\xi)=\xi$. Then
 $$|f'(\xi)|>1$$
unless $f(q)=q$ for all $q\in\mathbb B$.
\end{thm}

An improved  version  of Theorem \ref{Generalized Herzig} for  $\xi=1$ is called Hopf's lemma, whose proof  is based on Theorem \ref{Julia-Caratheodory}.

\begin{thm}\label{Herzig}{\bf(Hopf)}
Let $f$ be a slice regular function on $\mathbb B\cup\{1\}$ such that $f(\mathbb B)\subseteq\mathbb B$ and $f(1)=1$. Then the following statements hold true:

\begin{enumerate} \item[(i)]
The derivative of $f$ at $1$ is real and
$$f'(1)\geq\frac{2|1-f(0)|^2}{1-|f(0)|^2+|f'(0)|}.$$
In particular, $$f'(1) \geq\frac{|1-f(0)|^2}{1-|f(0)|^2}.$$
Moreover, equality holds for the last  inequality if and only if $$f(q)=(1-q\bar{u})^{-\ast}\ast(q-u)(1-\bar{u})(1-u)^{-1}$$ for some point $u\in\mathbb B$.
\item[(ii)]
If further  $$f^{(k)}(0)=0, \qquad \forall\ k=0,1,\ldots,n-1$$
for some  $n\in\mathbb N$, then
 $$f'(1)\geq n+\frac{2 \big|1-f^{(n)}(0)/n!\big|^2}{1-\big|f^{(n)}(0)/n!\big|^2
 +\big|f^{(n+1)}(0)\big|/(n+1)!}.$$
In particular,
$$f'(1)\geq n+\frac{\big|1-f^{(n)}(0)/n!\big|^2}{1-\big|f^{(n)}(0)/n!\big|^2}.$$
Moreover, equality holds for the last inequality if and only if $$f(q)=q^n(1-q\bar{u})^{-\ast}\ast(q-u)(1-\bar{u})(1-u)^{-1}$$ for some point $u\in\mathbb B$.
 \end{enumerate}
\end{thm}

Notice that  the Julia-Carath\'{e}odory theorem in Theorem \ref{Julia-Caratheodory} holds only for the real boundary points $\xi=\pm1\in\partial\mathbb B$.
 As shown by
Theorem \ref{Generalized Herzig},  the relation $$f'(\xi)=\alpha \overline{\xi}f(\xi)$$
 does  no longer hold in general in  the setting of quaternions  under the condition that $$\alpha:=\liminf\limits_{\mathbb B\ni q\rightarrow \xi}\dfrac{1-|f(q)|}{1-|q|}<+\infty$$  in contrast to the complex setting.
Consequently, the general Julia-Carath\'{e}odory theorem (the case that $\pm1\neq\xi\in\partial\mathbb B$) will be much more delicate and requires further research. At the same time, this phenomenon reflects fully the special role of the real axis in the theory of slice regular functions.

The outline of this paper is as follows. In Section \ref{Preliminaries}, we set up basic notations and give some preliminary results. In Section \ref{Proofs of Main Theorems}, we give the detailed proofs of main results for slice regular self-mappings of the open unit ball $\mathbb B$. To this end, we shall establish the Lindel\"{o}f principle and the Lindel\"{o}f inequality. The analogous results for slice regular self-mappings of the right half-space $\mathbb H^+$ are established in Section \ref{Julia-Caratheodory theorem}, of which the starting point is the right half-space version of the Schwarz-Pick theorem.

\section{Preliminaries}\label{Preliminaries}

We recall in this section some preliminary definitions and results on slice regular functions.
To have a more complete insight on the theory, we refer the reader to the monograph \cite{GSS}.

Let $\mathbb H$ denote the non-commutative, associative, real algebra of quaternions with standard basis $\{1,\,i,\,j, \,k\}$,  subject to the multiplication rules
$$i^2=j^2=k^2=ijk=-1.$$
 Every element $q=x_0+x_1i+x_2j+x_3k$ in $\mathbb H$ is composed by the \textit{real} part ${\rm{Re}}\, (q)=x_0$ and the \textit{imaginary} part ${\rm{Im}}\, (q)=x_1i+x_2j+x_3k$. The \textit{conjugate} of $q\in \mathbb H$ is then $\bar{q}={\rm{Re}}\, (q)-{\rm{Im}}\, (q)$ and its \textit{modulus} is defined by $|q|^2=q\overline{q}=|{\rm{Re}}\, (q)|^2+|{\rm{Im}}\, (q)|^2$. We can therefore calculate the multiplicative inverse of each $q\neq0$ as $ q^{-1} =|q|^{-2}\overline{q}$.
 Every $q \in \mathbb H $ can be expressed as $q = x + yI$, where $x, y \in \mathbb R$ and
$$I=\dfrac{{\rm{Im}}\, (q)}{|{\rm{Im}}\, (q)|}$$
 if ${\rm{Im}}\, q\neq 0$, otherwise we take $I$ arbitrarily such that $I^2=-1$.
Then $I $ is an element of the unit 2-sphere of purely imaginary quaternions
$$\mathbb S=\big\{q \in \mathbb H:q^2 =-1\big\}.$$

For every $I \in \mathbb S $ we will denote by $\mathbb C_I$ the plane $ \mathbb R \oplus I\mathbb R $, isomorphic to $ \mathbb C$, and, if $\Omega \subseteq \mathbb H$, by $\Omega_I$ the intersection $ \Omega \cap \mathbb C_I $. Also, for $R>0$, we will denote the open ball centred at the origin with radius $R$ by
$$B(0,R)=\big\{q \in \mathbb H:|q|<R\big\}.$$

We can now recall the definition of slice regularity.
\begin{defn} \label{de: regular} Let $\Omega$ be a domain in $\mathbb H$. A function $f :\Omega \rightarrow \mathbb H$ is called \emph{slice} \emph{regular} if, for all $ I \in \mathbb S$, its restriction $f_I$ to $\Omega_I$ is \emph{holomorphic}, i.e., it has continuous partial derivatives and satisfies
$$\bar{\partial}_I f(x+yI):=\frac{1}{2}\left(\frac{\partial}{\partial x}+I\frac{\partial}{\partial y}\right)f_I (x+yI)=0$$
for all $x+yI\in \Omega_I $.
 \end{defn}
As shown in \cite {CGSS}, the natural   domains of definition  of slice regular functions are the  so-called axially symmetric slice domains.

\begin{defn} \label{de: domain}
Let $\Omega$ be a domain in $\mathbb H $.  $\Omega$ is called a \textit{slice domain}  if $\Omega$ intersects the real axis and $\Omega_I$  is a domain
of $ \mathbb C_I $  for any $I \in \mathbb S $.

Moreover,  if  $x + yI \in \Omega$ implies $x + y\mathbb S \subseteq \Omega $ for any $x,y \in \mathbb R $ and $I\in \mathbb S$, then
 $\Omega$  is called an \textit{axially symmetric slice domain}.
\end{defn}

From now on, we will omit the term `slice' when referring to slice regular functions and will focus mainly on regular functions on $B(0,R)=\big\{q \in \mathbb H:|q|<R\big\}$ and the right half-space $\mathbb H^+=\big\{q\in\mathbb H: \textrm{Re}(q)>0\big\}.$
For regular functions the natural definition of derivative is given by the following (see \cite{GS1,GS2}).
\begin{defn} \label{de: derivative}
Let $f :B(0,R) \rightarrow \mathbb H$  be a regular function. For each $I\in\mathbb S$, the $I$- derivative of $f$ at $q=x+yI$
is defined by
$$\partial_I f(x+yI):=\frac{1}{2}\left(\frac{\partial}{\partial x}-I\frac{\partial}{\partial y}\right)f_I (x+yI)$$
on $\Omega_I$. The slice derivative of $f$ is the function $f'$ defined by $\partial_I f$ on $\Omega_I$ for all $I\in\mathbb S$.
 \end{defn}
The definition is well-defined because, by direct calculation, $\partial_I f=\partial_J f$ in $\Omega_I\cap \Omega_J$ for any choice of $I$, $J\in\mathbb S$. Furthermore, notice that the operators $\partial_I$ and $\bar{\partial}_I $ commute, and $\partial_I f=\frac{\partial f}{\partial x}$ for regular functions. Therefore, the slice derivative of a regular function is still regular so that we can iterate the differentiation to obtain the $n$-th
slice derivative
$$\partial^{n}_I f=\frac{\partial^{n} f}{\partial x^{n}},\quad\,\forall \,\, n\in \mathbb N. $$

In what follows, for the sake of simplicity, we will denote the $n$-th slice derivative by $f^{(n)}$ for every $n\in \mathbb N$. Incidentally, the slice derivative $f'$ is initially called Cullen derivative in \cite{GS1,GS2} and is denoted by $\partial_Cf$ due to the work of Cullen \cite{Cullen}. Here we follow the standard notations and terminology in the monograph \cite{GSS}.

In the theory of regular functions,  the following splitting lemma (see \cite{GS2}) relates closely slice regularity to classical holomorphy.
\begin{lem}{\bf(Splitting Lemma)}\label{eq:Splitting}
Let $f$ be a regular function on $B = B(0,R)$. Then for any $I\in \mathbb S$ and any $J\in \mathbb S$ with $J\perp I$, there exist two holomorphic functions $F,G:B_I\rightarrow \mathbb C_I$ such that
$$f_I(z)=F(z)+G(z)J,\qquad \forall\ z=x+yI\in B_I.$$
\end{lem}

%
%

Since the regularity does not keep under pointwise product of two regular functions a new multiplication operation, called the regular product (or $\ast$-product), appears via  a suitable modification of the usual one subject to noncommutative setting. The regular product plays a key role in the theory of slice regular functions. On open balls centred at the origin, the $\ast$-product of two regular functions is defined by means of their power series
expansions (see, e.g., \cite{GSS1,CGSS}).

\begin{defn}\label{R-product}
Let $f$, $g:B=B(0,R)\rightarrow \mathbb H$ be two regular functions and let
$$f(q)=\sum\limits_{n=0}^{\infty}q^na_n,\qquad g(q)=\sum\limits_{n=0}^{\infty}q^nb_n$$
be their series expansions. The regular product (or $\ast$-product) of $f$ and $g$ is the function defined by
$$f\ast g(q)=\sum\limits_{n=0}^{\infty}q^n\bigg(\sum\limits_{k=0}^n a_kb_{n-k}\bigg)$$
and it is regular on $B$.
\end{defn}
Notice that the $\ast$-product is associative and is not, in general, commutative. Its connection with the usual pointwise product is clarified by the following result \cite{GSS1,CGSS}.
\begin{prop} \label{prop:RP}
Let $f$ and $g$ be  regular on $B=B(0,R)$. Then for all $q\in B$,
$$f\ast g(q)=
\left\{
\begin{array}{lll}
f(q)g(f(q)^{-1}qf(q)) \qquad \,\,if \qquad f(q)\neq 0;
\\
\qquad  \qquad  0\qquad  \qquad \qquad if \qquad f(q)=0.
\end{array}
\right.
$$
\end{prop}
We remark that if $q=x+yI$ and $f(q)\neq0$, then $f(q)^{-1}qf(q)$ has the same modulus and same real part as $q$. Therefore $f(q)^{-1}qf(q)$ lies in the same 2-sphere $x+y\mathbb S$ as $q$. Notice that a zero $x_0+y_0I$ of the function $g$ is not necessarily a zero of $f\ast g$, but some element on the same sphere $x_0+y_0\mathbb S$ does. In particular, a
real zero of $g$ is still a zero of $f\ast g$. To present a characterization of the structure of the zero set of a regular function $f$, we need to introduce the following functions.
\begin{defn} \label{de: R-conjugate}
Let $ f(q)=\sum\limits_{n=0}^{\infty}q^na_n $ be a regular function on $B=B(0,R)$. We define the \emph{regular conjugate} of $f$ as
$$f^c(q)=\sum\limits_{n=0}^{\infty}q^n\bar{a}_n,$$
and the \emph{symmetrization} of $f$ as
$$f^s(q)=f\ast f^c(q)=f^c\ast f(q)=\sum\limits_{n=0}^{\infty}q^n\bigg(\sum\limits_{k=0}^n a_k\bar{a}_{n-k}\bigg).$$
Both $f^c$ and $f^s$ are regular functions on $B$.
\end{defn}
We are now able to define the inverse element of a regular function $f$ with respect to the $\ast$-product. Let $\mathcal{Z}_{f^s}$ denote the zero set of the symmetrization $f^s$ of $f$.
\begin{defn} \label{de: R-Inverse}
Let $f$ be a regular function on $B=B(0,R)$. If $f$ does not vanish identically, its \emph{regular reciprocal} is the function defined by
$$f^{-\ast}(q):=f^s(q)^{-1}f^c(q)$$
and it is regular on $B \setminus \mathcal{Z}_{f^s}$.

\end{defn}
The following result shows that the regular quotient is nicely related to the pointwise quotient (see \cite{Stop1, Stop2}).
\begin{prop} \label{prop:Quotient Relation}
Let $f$ and $g$ be regular  on  $B=B(0,R)$. Then for all $q\in B \setminus \mathcal{Z}_{f^s}$, $$f^{-\ast}\ast g(q)=f(T_f(q))^{-1}g(T_f(q)),$$
where $T_f:B \setminus \mathcal{Z}_{f^s}\rightarrow B \setminus \mathcal{Z}_{f^s}$ is defined by
$T_f(q)=f^c(q)^{-1}qf^c(q)$. Furthermore, $T_f$ and $T_{f^c}$ are mutual inverses so that $T_f$ is a diffeomorphism.
\end{prop}

Let us set
$$U(x_0+y_0\mathbb S, R)=\big\{q\in\mathbb H:\big|(q-x_0)^2+y_0^2\big|<R^2\big\}$$
for all $x_0, y_0\in\mathbb R$ and all $R>0$. The following result was proved in \cite{Stop3}; see Theorems 4.1 and 6.1 there for more details.
\begin{thm}\label{direction derivative}
Let $f$ be a regular function on a symmetric slice domain $\Omega$, and let $q_0=x_0+Iy_0\in U(x_0+y_0\mathbb S, R)\subseteq\Omega$. Then there exists $\{A_n\}_{n\in\mathbb N}\subset \mathbb H$ such that
 \begin{equation}\label{150}
 f(q)=\sum\limits_{n=0}^{\infty}
 \big((q-x_0)^2+y_0^2\big)^n\big(A_{2n}+(q-q_0)A_{2n+1}\big)
 \end{equation}
 for all $q\in U(x_0+y_0\mathbb S, R).$
 \end{thm}

 As a consequence, for all $v\in\mathbb H$ with $|v|=1$ the directional derivative of $f$ along $v$ can be computed at $q_0$ as
$$\frac{\partial f}{\partial v}(q_0)=\lim\limits_{t\rightarrow 0}\frac{f(q_0+tv)-f(q_0)}{t}=vA_1+(q_0v-v\overline{q_0})A_2,$$
where $$A_1=R_{q_0}f(\overline{q_0})=\partial_s f(q_0), \qquad A_2=R_{\overline{q_0}}R_{q_0}f(q_0).$$ 
In particular, there holds that
$$f'(q_0)=R_{q_0}f(q_0)=A_1+2\,{\rm{Im}}(q_0)A_2.$$

\section{Proofs of Main Theorems}\label{Proofs of Main Theorems}
In this section, we shall give the proofs of Theorems  \ref{Julia}-\ref{Herzig} except that of Theorem 3, which will be given in the next section.

\begin{proof}[Proof of Theorem $\ref{Julia}$]
The Schwarz-Pick Theorem shows, for all $q\in\mathbb B$,
 $$\big|\big(1-f(q)\overline{f(q_n)}\big)^{-\ast}\ast\big(f(q)-f(q_n)\big)\big|\leq \big|\big(1-q\overline{q_n}\big)^{-\ast}\ast\big(q-q_n\big)\big|,$$
which together with Proposition $\ref{prop:Quotient Relation}$ implies that
$$\frac{\big|f\circ T_{1-f\overline{f(q_n)}}(q)-f(q_n)\big|}{\big|1-f\circ T_{1-f\overline{f(q_n)}}(q)\overline{f(q_n)}\big|} \leq
\frac{\big|T_{1-Id\overline{q_n}}(q)-q_n\big|}
{\big|1-T_{1-Id\overline{q_n}}(q)\overline{q_n}\big|}
.$$
We square both sides  and then  minus by one to yield
$$\frac{\Big|1-f\circ T_{1-f\overline{f(q_n)}}(q)\overline{f(q_n)}\Big|^2}
{1-\big|f\circ T_{1-f\overline{f(q_n)}}(q)\big|^2}
\leq\frac{\big|1-T_{1-Id\overline{q_n}}(q)\overline{q_n}\big|^2}{1-|q|^2}\,\,
\frac{1-\big|f(q_n)\big|^2}{1-|q_n|^2}.$$
Letting $n\rightarrow\infty$, we obtain that
\begin{equation}\label{100}
\frac{\big|1-f\circ T_{1-f\overline{\eta}}(q)\bar{\eta}\big|^2}
{1-\big|f\circ T_{1-f\overline{\eta}}(q)\big|^2}
\leq \alpha\, \frac{\big|1-T_{1-Id\bar{\xi}}(q)\bar{\xi}\big|^2}{1-|q|^2}.
\end{equation}
It implies that $\alpha>0$ and
$$\frac{1-\big|f\circ T_{1-f\overline{\eta}}(q)\big|^2}{\big|1-f\circ T_{1-f\overline{\eta}}(q)\bar{\eta}\big|^2}\geq \frac{1}{\alpha}\, \frac{1-|q|^2}{\big|1-T_{1-Id\bar{\xi}}(q)\bar{\xi}\big|^2}.$$
This is equivalent to the inequality
$${\rm{Re}}\Big(\big(1-f\circ T_{1-f\overline{\eta}}(q)\overline{\eta}\big)^{-1}
\big(1+f\circ T_{1-f\overline{\eta}}(q)\overline{\eta}\big)\Big)
\geq \frac{1}{\alpha}\,{\rm{Re}}\Big(\big(1-T_{1-Id\overline{\xi}}(q)\overline{\xi}\big)^{-1}
\big(1+T_{1-Id\overline{\xi}}(q)\overline{\xi}\big)\Big).$$
That is, in terms of the regular product,
$${\rm{Re}}\Big(\big(1-f(q)\bar{\eta}\big)^{-\ast}\ast\big(1+f(q)\bar{\eta}\big)\Big)
\geq \frac{1}{\alpha}\,{\rm{Re}}\Big((1-q\bar{\xi})^{-\ast}\ast(1+q\bar{\xi})\Big).$$

In particular, when $\xi=1$, inequality $(\ref{100})$ becomes
\begin{equation}\label{101}
\frac{\big|1-f\circ T_{1-f\overline{\eta}}(q)\bar{\eta}\big|^2}
{1-\big|f\circ T_{1-f\overline{\eta}}(q)\big|^2}
\leq \alpha\, \frac{\big|1-q\big|^2}{1-|q|^2}
\end{equation}
due to the fact that $$\big|1-T_{1-Id}(q)\big|=\big|1-q\big|.$$
By Proposition $\ref{prop:Quotient Relation}$, $T_{1-f\bar{\eta}}$ is a homeomorphism with inverse $T_{1-\eta\ast f^c}$ since $f(\mathbb B)\subseteq\mathbb B$. Replacing $q$ by $T_{1-\eta\ast f^c}(q)$ in  inequality (\ref{101}) gives that
$$\frac{\big|\eta-f(q)\big|^2}
{1-\big|f(q)\big|^2}
\leq \alpha\, \frac{\big|1-T_{1-\eta\ast f^c}(q)\big|^2}{1-|q|^2}
=\alpha\, \frac{|1-q|^2}{1-|q|^2}, \qquad \forall\,\,q\in \mathbb B.$$

%

If equality holds for  Julia's inequality $(\ref{eq:11})$ at some point $q_0\in\mathbb B$, then   the function
$$\big(1-f(q)\bar{\eta}\big)^{-\ast}\ast\big(1+f(q)\bar{\eta}\big)-
\frac{1}{\alpha}(1-q\bar{\xi})^{-\ast}\ast(1+q\bar{\xi})$$
is an imaginary constant, say $It_0$ in virtue of the maximum principle for real part of regular functions (see Lemma 2 in \cite{RW1}). A simple calculation shows that
$$f(q)=\bigg(1+\frac{1}{\alpha}(1-q\bar{\xi})^{-\ast}\ast(1+q\bar{\xi})+It_0\bigg)^{-\ast}
\ast\bigg(\frac{1}{\alpha}(1-q\bar{\xi})^{-\ast}\ast(1+q\bar{\xi})+It_0-1\bigg)\eta.$$
Notice that the term in the first brackets can be written as
$$\frac1{\alpha}(1-q\bar{\xi})^{-\ast}\ast\big(1+\alpha+I\alpha t_0\big)
\ast(1+q\bar{u}),$$ where $$u=\big((1-\alpha)+I\alpha t_0\big)\xi\big((1+\alpha)-I\alpha t_0\big)^{-1}\in\mathbb B,$$ since $\alpha>0$. The other term can be treated similarly. Consequently, $f$ can be represented as
$$f(q)=(1+q\bar{u})^{-*}*(q+u)v,$$
where $$v=\big((1+\alpha)+I\alpha t_0\big)^{-1}\bar{\xi}\big((1+\alpha)-I\alpha t_0\big)\eta\in \partial\mathbb B.$$  It follows that the equality in Julia's inequality can hold only for regular M\"{o}bius transformations of $\mathbb B$ onto $\mathbb B$, and a direct calculation shows that it does indeed hold for all such regular M\"{o}bius transformations.  Now the proof is complete.
\end{proof}

To prove Theorem \ref{Julia-Caratheodory}, we shall need a quaternionic version of Lindel\"{o}f's principle, which  follows  easily from the corresponding  result in the complex setting and the splitting lemma.
\begin{lem}\label{Lindelof}{\bf(Lindel\"{o}f)}
Let $f$ be a regular function on $\mathbb B$ and bounded in each non-tangential approach region at $1$. If for some continuous curve $\gamma\in \mathbb B\cap\mathbb C_I$ ending at $1$ for some $I\in \mathbb S$, there exists the limit $$\eta=\lim\limits_{t\rightarrow1^-}f(\gamma(t)),$$
then $f$ also has the non-tangential limit $\eta$ at $1$.
\end{lem}
Now we are in a position to prove the Julia-Carath\'{e}odory theorem.
\begin{proof}[Proof of Theorem $\ref{Julia-Caratheodory}$]
The equivalence between  $(\textrm{ii})$ and $(\textrm{iii})$  follows directly from the corresponding  result in the complex setting (cf. \cite[p. 46, (VI-1)]{Sarason2}) and the splitting lemma.

 The implication $(\textrm{ii})\Rightarrow(\textrm{i})$ follows from the inequality
$$\frac{1-|f(r)|}{1-r}\leq \frac{|\eta-f(r)|}{1-r}, \qquad \forall\ r\in (0, 1).$$

Now we prove the implication $(\textrm{i})\Rightarrow (\textrm{ii})$. Under assumption $(\textrm{i})$, there exists a sequence $\{q_n\}_{n\in \mathbb N}\subset \mathbb B$ converging to $1$ as $n$ tends to $\infty$, such that
$$\alpha=\lim\limits_{n\rightarrow\infty}\frac{1-|f(q_n)|}{1-|q_n|}$$
and
$$\lim\limits_{n\rightarrow\infty}f(q_n)=\eta$$ for some $\eta\in \partial\mathbb B$. It follows from Julia's inequality (\ref{eq:10}) that
$$\frac{|\eta-f(q)|^2}{1-|f(q)|^2}\leq \alpha\,\frac{|1-q|^2}{1-|q|^2}.$$
Fix a non-tangential approach region at $1$, say the region
\begin{equation}\label{def:non-tangential-R}
\mathcal{R}(1,k)=\{q\in \mathbb B:|q-1|<k(1-|q|)\},
\end{equation}
where $k$ is a constant greater than one. For all $q\in\mathcal{R}(1,k)$ we have
\begin{eqnarray}\label{11}
\frac{|\eta-f(q)|^2}{1-|f(q)|^2}\leq \alpha k|q-1|\frac{(1-|q|)}{1-|q|^2}\leq \alpha k|q-1|,
\end{eqnarray}
which implies that $f(q)$ tends to $\eta$ as $q$ tends to $1$ within $\mathcal{R}(1,k)$.
In other words, $f$ has a non-tangential limit $\eta$ at $1$.

It remains to prove  that the difference quotient $$(q-1)^{-1}\big(f(q)-\eta\big)$$ has a non-tangential limit $\alpha\eta$. To this end, notice that
$$\frac{|\eta-f(q)|}{1+|f(q)|}\leq \frac{|\eta-f(q)|^2}{1-|f(q)|^2},$$
from which and inequality $(\ref{11})$ we have
$$\big|(q-1)^{-1}(f(q)-\eta)\big|\leq \alpha k\big(|1+|f(q)|\big)\leq 2\alpha k,$$
whenever $q\in\mathcal{R}(1,k)$. Consequently, the regular function $$g(q):=(q-1)^{-1}(f(q)-\eta)$$ is bounded in each non-tangential approach region at $1$.
The Lindel\"{o}f's principle in Lemma \ref{Lindelof} thus reduces the proof to the existence of the radial limit
$$\lim\limits_{r\rightarrow 1^-}\frac{\eta-f(r)}{1-r}=\alpha \eta.$$

To consider this radial limit, we observe from the definition of $\alpha$ as the lower limit in $(\textrm{i})$ that
\begin{equation}\label{12}
\liminf\limits_{r\rightarrow 1^-}\frac{1-|f(r)|}{1-r}\geq \alpha.
\end{equation}
On the other hand, setting $q=r$ in the Julia's inequality (\ref{eq:10}) yields that
\begin{equation*}\label{13}
\frac{1-|f(r)|}{1-r}\frac{1+r}{1+|f(r)|}\leq \frac{|\eta-f(r)|^2}{1-|f(r)|^2}\,\frac{1+r}{1-r}\leq \alpha \frac{(1-r)^2}{1-r^2}\,\frac{1+r}{1-r}.
\end{equation*}
Since $f$ has a non-tangential limit $\eta\in \partial \mathbb B$ at $1$, it follows that
\begin{equation}\label{14}
\limsup\limits_{r\rightarrow 1^-}\frac{1-|f(r)|}{1-r}\leq \alpha.
\end{equation}
Consequently,
\begin{equation}\label{15}
\lim\limits_{r\rightarrow 1^-}\frac{1-|f(r)|}{1-r}= \alpha.
\end{equation}

Furthermore,
$$\bigg(\frac{1-|f(r)|}{1-r}\bigg)^2\leq\bigg(\frac{|\eta-f(r)|}{1-r}\bigg)^2
\leq\alpha\frac{1-|f(r)|^2}{1-r^2}=\alpha\frac{1-|f(r)|}{1-r}\,\frac{1+|f(r)|}{1+r},$$
which together with $(\ref{15})$ implies that
\begin{equation}\label{16}
\lim\limits_{r\rightarrow 1^-}\frac{|\eta-f(r)|}{1-r}= \alpha.
\end{equation}
By $(\ref{15})$ and $(\ref{16})$, we have
\begin{equation}\label{17}
\lim\limits_{r\rightarrow 1^-}\frac{1-|f(r)|}{|\eta-f(r)|}=1.
\end{equation}
Since
$$\frac{1-|f(r)|}{|\eta-f(r)|}\leq \frac{1-{\rm{Re}}\big(f(r)\bar{\eta}\big)}{|1-f(r)\bar{\eta}|}\leq1,$$
it follows from  that
$$\lim\limits_{r\rightarrow 1^-}{\rm{Re}} \frac{1-f(r)\bar{\eta}}{|1-f(r)\bar{\eta}|}=1.$$
This forces that
$$\lim\limits_{r\rightarrow 1^-}\frac{1-f(r)\bar{\eta}}{|1-f(r)\bar{\eta}|}=1.$$
Therefore,
\begin{equation}\label{18}
\begin{split}
\lim\limits_{r\rightarrow 1^-}\frac{\eta-f(r)}{1-r}
&=\lim\limits_{r\rightarrow 1^-}\frac{1-f(r)\bar{\eta}}{1-r}\eta\\
&=\lim\limits_{r\rightarrow 1^-}\frac{1-f(r)\bar{\eta}}{|1-f(r)\bar{\eta}|}\,\frac{1-|f(r)|}{1-r}\,
\frac{|\eta-f(r)|}{1-|f(r)|}\eta\\
&=\alpha \eta,
\end{split}
\end{equation}
which completes the proof of the implication $(\textrm{i})\Rightarrow (\textrm{ii})$.

Now we assume that all conditions (i)-(iii) hold.
By carefully checking the above proof, we see that assertions (a)-(c) hold true.
It remains to verify (d).

From assertions (i), (ii), and (c), we know that $\alpha<\infty$ and  the difference quotient $$(q-1)^{-1}\big(f(q)-\eta\big)$$ has a non-tangential limit $\alpha\eta$ at point $1$. Let us set
$$g(q)=(q-1)^{-1}\big(f(q)\bar\eta-1\big)-\alpha.$$
Then $g$ is regular on $\mathbb B$ and has the non-tangential limit $0$ at  point $1$. Since
$$|f(q)|^2=\big|1+(q-1)(\alpha+g(q))\big|^2,$$
it follows that
\begin{equation}\label{19}
\frac{1-|f(q)|^2}{1-|q|^2}=2\alpha \frac{{\rm{Re}}(1-q)}{1-|q|^2}+2\frac{{\rm{Re}}\big((1-q)g(g)\big)}{1-|q|^2}
-\frac{|1-q|^2}{1-|q|^2}\big|\alpha+g(q)\big|^2.
\end{equation}
Fix the non-tangential approach region $\mathcal{R}(1,k)$ as in (\ref{def:non-tangential-R}) and consider the non-tangential limit as
 $q\rightarrow 1$ within $\mathcal{R}(1,k)$.
  It is clear that  the second term on the right-hand side of the preceding equality approaches $0$ and so does the third term. Since the equality
$$\frac{{\rm{Re}}(1-q)}{1-|q|^2}=\frac12\bigg(1+\frac{|1-q|^2}{1-|q|^2}\bigg)$$
shows that $$\mathop{\lim_{q\rightarrow 1}}_{q\in \mathcal{R}(1,k)}\frac{{\rm{Re}}(1-q)}{1-|q|^2}=\frac12,$$
from which it follows that the first term on the right-hand side of  equality $(\ref{19})$ approaches to $\alpha$ as $q\rightarrow 1$ within $\mathcal{R}(1,k)$. Therefore,
$$\mathop{\lim_{q\rightarrow 1}}_{q\in \mathcal{R}(1,k)}\frac{1-|f(q)|^2}{1-|q|^2}=\alpha,$$
which is equivalent to
$$\mathop{\lim_{q\rightarrow 1}}_{q\in \mathcal{R}(1,k)}\frac{1-|f(q)|}{1-|q|}=\alpha.$$
Now the proof is complete.
\end{proof}

Next we come to prove the quaternionic version of Hopf's theorem.

\begin{proof}[Proof of Theorem $\ref{Herzig}$]

 $(\textrm{i})$\ Let $f$ be as described in Theorem \ref{Herzig}. Set
$$g(q):=\big(1-f(q)\overline{f(0)}\big)^{-\ast}\ast\big(f(q)-f(0)\big)
(1-\overline{f(0)})(1-f(0))^{-1},$$ which is a regular function on $\mathbb B\cup\{1\}$ such that $g(\mathbb B)\subseteq\mathbb B$, $g(0)=0$ and $g(1)=1$. Moreover, an easy calculation shows that
\begin{equation}\label{ineq:i1}
f'(1)=\frac{|1-f(0)|^2}{1-|f(0)|^2}\,g'(1),\qquad |g'(0)|=\frac{|f'(0)|}{1-|f(0)|^2}\leq 1.
\end{equation}
Applying  Julia-Carath\'{e}odory Theorem and Julia's inequality (\ref{eq:10}) to the regular function $h(q):=q^{-1}g(q)$ mapping  $\mathbb B$  to $\overline{\mathbb B}$ yields that
$$  g'(1)=1+h'(1)\geq 1+\frac{|1-g'(0)|^2}{1-|g'(0)|^2}. $$
In particular,
\begin{equation}\label{ineq:i2}
  g'(1)\geq\frac{2}{1+|g'(0)|},
\end{equation}
which is also established easily from the second inequality in (\ref{23}). Substituting equalities in (\ref{ineq:i1}) to (\ref{ineq:i2}) yields that
$$f'(1)\geq\frac{2|1-f(0)|^2}{1-|f(0)|^2+|f'(0)|}.$$
In particular,
\begin{equation}\label{ineq:i3}
f'(1) \geq\frac{|1-f(0)|^2}{1-|f(0)|^2}.
\end{equation}
If equality holds for the last inequality, then equality holds for Julia's equality (\ref{eq:10}) at point $q=0$, it follows from Theorem \ref{Julia} and the assumption that $f(1)=1$ that
 \begin{equation}\label{ineq:i4}
 f(q)=(1-q\bar{u})^{-\ast}\ast(q-u)(1-\bar{u})(1-u)^{-1}
\end{equation}
for some point $u\in\mathbb B$. Therefore, the equality in inequality (\ref{ineq:i3}) can hold only for regular M\"{o}bius transformations of  the form (\ref{ineq:i4}), and a direct calculation shows that it does indeed hold for all such regular M\"{o}bius transformations.
Now the proof of $\textrm{(i)}$ is complete.

$\textrm{(ii)}$  The result follows easily from $\textrm{(i)}$ by considering  the regular function $h(q):=q^{-n}f(q)$ and noticing  that
 $$h(0)=\frac{f^{(n)}(0)}{n!},\qquad h'(0)=\frac{f^{(n+1)}(0)}{(n+1)!}.$$

\end{proof}

%
%

To prove the boundary Schwarz lemma (Theorem \ref{Generalized Herzig}), we make the best use of the classical Hopf's lemma, which can be viewed as the real version of the  boundary Schwarz lemma.
We remark that, unlike in the complex setting,  the boundary Schwarz lemma  can not be simplified to the specific case that  $\xi=1$ because the theory of  regular composition is unavailable.

Before presenting the proof of  Theorem \ref{Generalized Herzig}, we need the following Lindel\"{o}f inequality.

\begin{prop}\label{L-inequality}{\bf(Lindel\"{o}f)}
Let $f$ be a  regular self-mapping of the open unit ball $\mathbb B$. Then for all  $q\in\mathbb B$ we have

\begin{equation}\label{20}
\Big|f(q)-\frac{1-|q|^2}{1-|q|^2|f(0)|^2}f(0)\Big|
\leq\frac{|q|\big(1-|f(0)|^2\big)}{1-|q|^2|f(0)|^2};
\end{equation}

\begin{equation}\label{21}
\frac{|f(0)|-|q|}{1-|q||f(0)|}\leq|f(q)|\leq \frac{|q|+|f(0)|}{1+|q||f(0)|};
\end{equation}
\begin{equation}\label{22}
\big|f(q)-f(0)\big|\leq \frac{|q|\big(1-|f(0)|^2\big)}{1-|q||f(0)|}.
\end{equation}

Moreover, if there exists a positive integer $n$ such that
$$f^{(k)}(0)=0, \qquad k=0,1,\ldots,n-1,$$
 then it holds that
\begin{equation}\label{23}
\frac{|f^{(n)}(0)|-n!|q|}{n!-|q||f^{(n)}(0)|}|q|^n\leq|f(q)|\leq \frac{n!|q|+|f^{(n)}(0)|}{n!+|q||f^{(n)}(0)|}|q|^n,\qquad \forall\,q\in\mathbb B.
\end{equation}

Equality holds for one of inequalities in $(\ref{20})-(\ref{22})$ at some point
$q_0\in \mathbb B\setminus\{0\}$ if and only if $f$ is a regular M\"{o}bius transformation of $\mathbb B$ onto $\mathbb B$, and equality holds for one of inequalities in $(\ref{23}) $ at some point $q_0\in \mathbb B\setminus\{0\}$  if and only if $f(q)=q^ng(q)$, where $g$ is a regular M\"{o}bius transformation of $\mathbb B$ onto $\mathbb B$ or a unimodular constant  $u\in\partial\mathbb B$.
\end{prop}
\begin{proof}
The Schwarz-Pick Theorem gives, for all $q\in\mathbb B$,
\begin{equation}\label{30}
\big|\big(1-f(q)\overline{f(0)}\big)^{-\ast}\ast\big(f(q)-f(0)\big)\big|\leq |q|,
\end{equation}
which together with Proposition $\ref{prop:Quotient Relation}$ implies that
$$\frac{\big|f\circ T_{1-f\overline{f(0)}}(q)-f(0)\big|}{\big|1-f\circ T_{1-f\overline{f(0)}}(q)\overline{f(0)}\big|} \leq |q|.$$
An easy calculation shows that the preceding inequality is equivalent to
$$\Big|f\circ T_{1-f\overline{f(0)}}(q)-\frac{1-|q|^2}{1-|q|^2|f(0)|^2}f(0)\Big|
\leq\frac{|q|\big(1-|f(0)|^2\big)}{1-|q|^2|f(0)|^2}.$$
Since $f(\mathbb B)\subseteq\mathbb B$, it follows from  Proposition $\ref{prop:Quotient Relation}$ that $T_{1-f\overline{f(0)}}$ is a homeomorphism with inverse $T_{1-f(0)\ast f^c}$.  Replacing $q$ by $T_{1-f(0)\ast f^c}(q)$ in the preceding inequality gives that
$$\Big|f(q)-\frac{1-|q|^2}{1-|q|^2|f(0)|^2}f(0)\Big|
\leq\frac{|q|\big(1-|f(0)|^2\big)}{1-|q|^2|f(0)|^2},\qquad \forall\,q\in\mathbb B.$$

If equality achieves at some point $0\neq q_0\in \mathbb B$ in the last inequality, then it also achieves at point $0\neq\widetilde{q_0}=T_{1-f(0)\ast f^c}(q_0)$  in inequality $(\ref{30})$. It thus follows from the Schwarz-Pick theorem that $f$ is a regular M\"{o}bius transformation of $\mathbb B$ onto $\mathbb B$.

Inequalities in $(\ref{21})$ and $(\ref{22})$ follow from $(\ref{20})$ as well as  the triangle inequality. Moreover, if $f^{(k)}(0)=0$, $k=0,1,\ldots,n-1$ for some positive integer $n$, then the function $g(q):=q^{-n}f(q)$ is also a regular function from $\mathbb B$ to $\overline{\mathbb B}$  in virtue of the maximum principle, and in addition we  have  $g(0)=f^{(n)}(0)/n!$.  Therefore, inequalities in $(\ref{23})$ follow from $(\ref{21})$. The proof is complete.
\end{proof}

We now come to  prove the boundary Schwarz lemma.

\begin{proof}[Proof of Theorem $\ref{Generalized Herzig}$]
By assumption, the slice subharmonic function $|f|^2$ attains its maximum at the boundary point $\xi\in \partial\mathbb B$ so that the directional derivative of $|f|^2$ along $\xi$ at the point $\xi$ satisfies that
\begin{equation}\label{209}
\frac{\partial |f|^2}{\partial \xi}(\xi)>0,
\end{equation}
 in virtue of the classical  Hopf lemma for subharmonic functions. Moreover,
\begin{equation}\label{210}
\frac{\partial |f|^2}{\partial \tau}(\xi)=0, \qquad \forall\, \,\tau\in T_{\xi}(\partial \mathbb B)\cong \mathbb R^3.
\end{equation}
Indeed, for any unit tangent vector $\tau\in T_{\xi}(\partial \mathbb B)$, take a smooth curve
$\gamma:(-1,1)\rightarrow\overline{\mathbb B}$ such that
$$\gamma(0)=\xi,\quad \gamma'(0)=\tau.$$
By definition we have
$$\frac{\partial |f|^2}{\partial \tau}(\xi)=\left.\Big(\frac{d}{dt}\big|f(\gamma(t))\big|^2\Big)\right|_{t=0}=0,$$
since the function $|f(\gamma(t))\big|^2$ in $t$ attains its maximum at the point $t=0$.

Moreover,   Theorem \ref{direction derivative} shows that
$$\frac{\partial f}{\partial v}(\xi)=v\partial_sf(\xi)+(\xi v-v\overline{\xi})R_{\bar{\xi}}R_{\xi}f(\xi)$$
 for all $v\in \mathbb H$ with $|v|=1$.
Therefore,
\begin{equation}\label{211}
\begin{split}
\frac{\partial |f|^2}{\partial v}(\xi)
&=2\,\Big\langle\frac{\partial f}{\partial v}(\xi),f(\xi)\Big\rangle\\
&=2\,\Big\langle v\partial_sf(\xi)+(\xi v-v\overline{\xi})R_{\bar{\xi}}R_{\xi}f(\xi),f(\xi)\Big\rangle\\
&=2\,\Big\langle v,f(\xi)\overline{\partial_sf(\xi)}+\bar{\xi}f(\xi)\overline{R_{\bar{\xi}}R_{\xi}f(\xi)}-
f(\xi)\overline{R_{\bar{\xi}}R_{\xi}f(\xi)}\xi\Big\rangle\\
&=2\,\Big\langle v,f(\xi)\overline{f'(\xi)}+\big[\bar{\xi}, f(\xi)\overline{R_{\bar{\xi}}R_{\xi}f(\xi)}\,\big]\Big\rangle,
\end{split}
\end{equation}
where $\langle$ , $\rangle$ denotes  the standard inner product on $\mathbb H\cong\mathbb R^4$, i.e.,$$\langle p,q\rangle=\textrm{Re}(p\bar{q}).$$
In the last equality we have used the fact that $$f'(\xi)=\partial_sf(\xi)+2\,{\rm{Im}}(\xi)R_{\bar{\xi}}R_{\xi}f(\xi).$$
Now it follows from (\ref{210}) and (\ref{211}) that
$$
f(\xi)\overline{f'(\xi)}+\big[\bar{\xi}, f(\xi)\overline{R_{\bar{\xi}}R_{\xi}f(\xi)}\,\big]\perp
 T_{\xi}(\partial \mathbb B)$$
so that in view of (\ref{209}) and (\ref{211})  there exists a real number $\lambda>0$ such that
$$f(\xi)\overline{f'(\xi)}+\big[\bar{\xi}, f(\xi)\overline{R_{\bar{\xi}}R_{\xi}f(\xi)}\,\big]=\lambda\,\xi$$
and
\begin{equation}\label{15110}
\overline{\xi}\Big(f(\xi)\overline{f'(\xi)}+\big[\bar{\xi}, f(\xi)\overline{R_{\bar{\xi}}R_{\xi}f(\xi)}\,\big]\Big)
=\lambda
=\frac12\frac{\partial |f|^2}{\partial \xi}(\xi)=\frac{\partial |f|}{\partial \xi}(\xi)>0.
\end{equation}

In order to get a better lower bound for the directional derivative $\dfrac{\partial |f|}{\partial \xi}(\xi)$, we make the best use of the so-called Lindel\"{o}f's inequality. Set
\begin{equation}\label{15111}
g(q):=\big(1-f(q)\overline{f(0)}\big)^{-\ast}\ast\big(f(0)-f(q)\big),
\end{equation}
which is a regular function on $\mathbb B\cup[\xi]$ such that $g(\mathbb B)\subseteq\mathbb B$ and $g(\xi)\in \partial\mathbb B$. Moreover, it is evident that $g(0)=0$ and
\begin{equation}\label{15112}
g'(0)=-\frac{f'(0)}{1-|f(0)|^2}.
\end{equation}
It follows from the second inequality in (\ref{23}) that
\begin{equation}\label{15113}
|g(q)|\leq \frac{|q|+|g'(0)|}{1+|q||g'(0)|}|q|,\qquad \forall\,q\in\mathbb B.
\end{equation}
From (\ref{15111}), we obtain that
$$f(q)=\big(1-g(q)\overline{f(0)}\big)^{-\ast}\ast\big(f(0)-g(q)\big).$$
This together with Proposition \ref{prop:Quotient Relation} implies
$$f(q)=\Big(1-g\circ T_{1-g\overline{f(0)}}(q)\overline{f(0)}\Big)^{-1}\Big(f(0)-g\circ T_{1-g\overline{f(0)}}(q)\Big),$$
from which it follows as in the complex setting that
\begin{equation}\label{15114}
 |f(q)|\leq \frac{|f(0)|+\big|g\circ T_{1-g\overline{f(0)}}(q)\big|}
{1+|f(0)|\big|g\circ T_{1-g\overline{f(0)}}(q)\big|},\qquad \forall\,q\in\mathbb B.
\end{equation}
Observe that the function $$\varphi_t(x)=\frac{x+t}{1+tx}$$
is increasing in $x\in [0,1]$ for any given $t\in [0,1]$, and $|T_{1-g\overline{f(0)}}(q)|=|q|$ for all $q\in\mathbb B$. Consequently, from inequalities (\ref{15113}) and (\ref{15114}) we conclude that
\begin{equation}\label{15115}
|f(q)|\leq \frac{|f(0)|+\mu_g(q)}
{1+|f(0)|\mu_g(q)},
\end{equation}
where $$\mu_g(q)=\frac{|q|+|g'(0)|}{1+|q||g'(0)|}|q|$$ for all $q\in\mathbb B$.
This together with an easy  calculation gives
$$\frac{1-|f(q)|}{1-|q|}\geq
\frac{\big(1-|f(0)|\big)\big(1+|q|\big)}{\big(1+|q||g'(0)|\big)\big(1+|f(0)|\mu_g(q)\big)},$$
which results in
$$\frac{\partial |f|}{\partial \xi}(\xi)
=\lim\limits_{r\rightarrow 1^{-}}\frac{1-|f(r\xi)|}{1-r}
\geq\frac{2\big(1-|f(0)|\big)}{\big(1+|g'(0)|\big)\big(1+|f(0)|\big)}
=\frac{2\big(1-|f(0)|\big)^2}{1-|f(0)|^2+|f'(0)|}.$$
Here the last equality follows from equality  (\ref{15112}). Now the proof of $\textrm{(i)}$ is complete and it remains to prove $\textrm{(ii)}$.

However, $\textrm{(ii)}$ follows easily from $\textrm{(i)}$ by considering  the regular function $h(q):=q^{-n}f(q)$ and noticing  that
 $$h(0)=\frac{f^{(n)}(0)}{n!},\qquad h'(0)=\frac{f^{(n+1)}(0)}{(n+1)!}.$$
Moreover,
$$f(\xi)\overline{f'(\xi)}+\big[\bar{\xi}, f(\xi)\overline{R_{\bar{\xi}}R_{\xi}f(\xi)}\,\big]
=n\xi+h(\xi)\overline{h'(\xi)}+\big[\bar{\xi}, h(\xi)\overline{R_{\bar{\xi}}R_{\xi}h(\xi)}\,\big]$$
as one easily verifies. Now the proof is complete.
\end{proof}

Some useful remarks are in order.
\begin{rem}\label{rem-Schwarz1} It is quite natural to ask if the quality
$$\overline{\xi}\Big(f(\xi)\overline{f'(\xi)}+\big[\bar{\xi}, f(\xi)\overline{R_{\bar{\xi}}R_{\xi}f(\xi)}\,\big]\Big)$$
in Theorem   $\ref{Generalized Herzig}$ is no other than
$$\overline{\xi}f(\xi)\overline{f'(\xi)}$$
as in the complex setting.
Unfortunately, the Lie bracket
$$\big[\bar{\xi}, f(\xi)\overline{R_{\bar{\xi}}R_{\xi}f(\xi)}\,\big]$$
in Theorem $\ref{Generalized Herzig}$ does not necessarily vanish in general. Moreover,
all the products of
$\xi$, $f(\xi)$ and $\overline{f'(\xi)}$ in any different orders  may fail simultaneously to be  real numbers so that the inequality
$$\overline{\xi} f(\xi)\overline{f'(\xi)}\geq \frac{2\big(1-|f(0)|\big)^2}{1-|f(0)|^2+|f'(0)|}$$
does not hold, neither all of its modified versions free of orders. These facts can be
demonstrated by  the following counterexample.
\end{rem}

\begin{ex}\label{ex-Schwarz}
Let $I\in\mathbb S$ be fixed. Set
$$f(q)=\big(1+qI/2\big)^{-\ast}\ast(q-I/2).$$
Then it is a slice regular M\"{o}bius transformation of $\mathbb B$ onto $\mathbb B$.
 It is evident to see that
 $$f(q)=\big(q^2+4\big)^{-1}\big(3q-2(q^2+1)I\big)$$
so  that it satisfies all the assumptions of Theorem $\ref{Generalized Herzig}$.

Now we set $\xi=J$, where $J\in\mathbb S$ is fixed  such that $J\perp I$.   An easy calculation thus shows that
$$f(J)=J,\qquad f'(J)=\frac13(5+4IJ),$$
and
$$A_1=\partial_sf(J)=1,\qquad A_2=R_{-J}R_{J}f(J)=-\frac13(2I+J).$$
Therefore,
$$\big[\,\overline{J},f(J)\overline{R_{-J}R_{J}f(J)}\,\big]
=-\frac13\big[J,J(2I+J)\big]=\frac43I\neq0.$$

On the other hand,
$$\overline{J}\Big(f(J)\overline{f'(J)}+
\big[\,\overline{J},f(J)\overline{R_{-J}R_{J}f(J)}\,\big]\Big)
=\frac53>\frac{2\big(1-|f(0)|\big)^2}{1-|f(0)|^2+|f'(0)|}=\frac13$$
as predicated by Theorem $\ref{Generalized Herzig}$.
\end{ex}

%

\bigskip

Now we provide an example to show that in Corollary \ref{Cor-Schwarz} the inequality $f'(\xi)>1$  may fail, or rather that
$f'(\xi)$ may not be a real number.
\begin{ex}\label{ex-Schwarz1} We now construct a function $g$ such that $g'(J)$ is no longer a real number. To this end, we set
$$g(q)=-qf(q)J=-q\big(1+qI/2\big)^{-\ast}\ast(q-I/2)J,$$
where  $f$ is as described in Example \ref{ex-Schwarz} and $I, J\in\mathbb S$ with $J\perp I$.

It is evident that this function is a  Blaschke product of order 2 so that
it is regular on $\overline{\mathbb B}$,
 and satisfies  $g(\mathbb B)\subseteq\mathbb B$, $g(0)=0$ and $g(J)=-Jf(J)J=J$.
This means that $g$ satisfies all assumptions given in Corollary \ref{Cor-Schwarz}.

  However, we find that $g'(J)$ is indeed not a real number.
  In fact, by  the Leibniz rule we have
$$g'(q)=-\big(f(q)+qf'(q)\big)J.$$
Consequently,
$$g'(J)=\frac43(2-IJ)\notin \mathbb R.$$

On the other hand,  a simple calculation shows that
$$\big[J,R_{-J}R_{J}g(J)\,\big]=-\frac43IJ$$
and
$$g'(J)-\big[J, R_{-J}R_{J}g(J)\,\big]=\frac83>1$$
as predicated by Corollary \ref{Cor-Schwarz}.
\end{ex}

\begin{rem}\label{rem-Schwarz2}
Theorem \ref{Generalized Herzig} also provides a lower bound for the slice derivative of $f$ at $\xi$ in modulus. Notice that
$$\frac{\partial f}{\partial \xi}(\xi)=\xi f'(\xi)$$ and
$$\frac{\partial |f|}{\partial \xi}(\xi)
=\overline{\xi}\Big(f(\xi)\overline{f'(\xi)}+\big[\bar{\xi}, f(\xi)\overline{R_{\bar{\xi}}R_{\xi}f(\xi)}\,\big]\Big)$$
from (\ref{15110}), thus the obvious inequality
$$\Big|\frac{\partial f}{\partial \xi}(\xi)\Big|\geq \frac{\partial |f|}{\partial \xi}(\xi)$$ shows that the left-hand side of each inequality in Theorem \ref{Generalized Herzig} can be replaced by $|f'(\xi)|$.
\end{rem}

\begin{rem}\label{rem-Schwarz3}
The proof of Theorem \ref{Generalized Herzig} contains implicitly the inequality
$$|f(q)|\leq \frac{|f(0)|+|q|\dfrac{|q|\big(1-|f(0)|^2\big)+|f'(0)|}{1-|f(0)|^2+|q||f'(0)|}}
{1+|f(0)||q|\dfrac{|q|\big(1-|f(0)|^2\big)+|f'(0)|}{1-|f(0)|^2+|q||f'(0)|}},\qquad \forall\, q\in \mathbb B,$$
for any regular function $f$ mapping $\mathbb B$ into $\overline{\mathbb B}$. Obviously, this inequality is more precise than that one in (\ref{21}).
\end{rem}

\section{Julia-Carath\'{e}odory theorem in  $\mathbb H^{+}$}
\label{Julia-Caratheodory theorem}
In this section, we establish the Julia-Carath\'{e}odory theorem for slice regular self-mappings of the right half-space
$$\mathbb H^{+}:=\big\{q\in \mathbb H:{\rm{Re}}(q)>0\big\}.$$
 The proof depends ultimately on the right half-space version of  the Schwarz-Pick theorem, whose proof is the same as the one in \cite{BS} and hence is omitted.
\begin{thm}\label{Schwarz-Pick}{\bf(Schwarz-Pick)}
Let $f:\mathbb H^+\rightarrow\mathbb H^+$ be a regular function. Then for any $q_0\in \mathbb H^+$ we have
$$\Big|\big(g(q)+\overline{g(q_0)}\big)^{-\ast}\ast\big(g(q)-g(q_0)\big)\Big|
\leq\big|(q+\overline{q_0})^{-\ast}\ast(q-q_0)\big|, \qquad \forall\,q\in\mathbb H^+.$$
Inequality is strict $($except at $q=q_0$$)$ unless $f$ is a regular M\"{o}bius transformation from $\mathbb H^{+}$ onto itself.
\end{thm}

For every  $\gamma\in (0,1)$, we denote by $\mathcal{S}_{\gamma}$ the non-tangential cone at the boundary point $0$ of $\mathbb H^+$ given by
$$\mathcal{S}_{\gamma}:=\big\{q\in\mathbb H^+:{\rm{Re}}(q)>\gamma |q|\big\}.$$

\begin{thm}\label{JC-half space}{\bf(Julia-Carath\'{e}odory)}
Let $f:\mathbb H^+\rightarrow\mathbb H^+$ be a regular function and set
$$c:=\inf\bigg\{\frac{{\rm{Re}}f(q)}{{\rm{Re}}(q)}:q\in\mathbb H^{+}\bigg\}\geq0.$$ Then  the following  hold true:
\begin{enumerate}
\item[(i)]  for any $q\in\mathbb H^+$, $$\quad {\rm{Re}}f(q)\geq c\,{\rm{Re}}(q);$$
\item[(ii)]
 for any $\gamma\in (0,1)$,  $$\mathop{\lim_{|q|\rightarrow\infty}}_{q\in \mathcal{S}_{\gamma}}q^{-1}f(q)
=\mathop{\lim_{|q|\rightarrow\infty}}_{q\in \mathcal{S}_{\gamma}}\frac{{\rm{Re}}f(q)}{{\rm{Re}}(q)}
=c;$$
\item[(iii)]  for any $\gamma\in (0,1)$,
 $$\mathop{\lim_{|q|\rightarrow\infty}}_{q\in \mathcal{S}_{\gamma}}f'(q)=c. $$
\end{enumerate}
\end{thm}

\begin{proof}
We put
\begin{equation}\label{def:g-c-45}
g(q):=f(q)-cq,\qquad \forall\,q\in\mathbb H^{+},
\end{equation}
Then by definition ${\rm{Re}}\,g(q)\geq0$ for all $q\in\mathbb H^{+}.$
Moreover, we may assume that
$${\rm{Re}}\,g(q)>0,\qquad \forall\,q\in\mathbb H^{+},$$
in virtue of the maximum principle for real parts of regular functions, see Lemma 2 in \cite{RW1}. Otherwise, $g(q)=It_0$ for some $t_0\in\mathbb R$ and some $I\in\mathbb S$. Thus the results are obvious.

It follows from the Schwarz-Pick theorem that for all $q,\,q_0\in\mathbb H^+$,
$$\Big|\big(g(q)+\overline{g(q_0)}\big)^{-\ast}\ast\big(g(q)-g(q_0)\big)\Big|
\leq\big|(q+\overline{q_0})^{-\ast}\ast(q-q_0)\big|,$$
which is equivalent to
$$\frac{\big| g\circ T_{g+\overline{g(q_0)}}(q)-g(q_0)\big|}
{\big|g\circ T_{g+\overline{g(q_0)}}(q)+\overline{g(q_0)}\big|}
\leq
\frac{\big|T_{Id+\overline{q_0}}(q)-q_0\big|}
{\big|T_{Id+\overline{q_0}}(q)+\overline{q_0}\big|}
,$$
or
\begin{equation}\label{212}
\frac{\big| g(q)-g(q_0)\big|}
{\big|g(q)+\overline{g(q_0)}\big|}
\leq \frac{\big|p-q_0\big|}{\big|p+\overline{q_0}\big|}=:r,
\end{equation}
where
\begin{equation}\label{def:p-T-23}
p=T_{Id+\overline{q_0}}\circ T_{g^c+g(q_0)}(q).
\end{equation}

We set $$h(q)=\frac{g(q)-{\rm{Im}}g(q_0)}{{\rm{Re}}g(q_0)}.$$
Then
(\ref{212}) becomes
\begin{equation}\label{212-1}
\frac{\big| h(q)-1\big|}
{\big|h(q)+1\big|}
\leq r=\frac{\big|p-q_0\big|}{\big|p+\overline{q_0}\big|}.
\end{equation}
Consequently,
\begin{equation}\label{213}
|h(q)|\leq\frac{1+r}{1-r}=\frac{(1+r)^2}{1-r^2}
=\frac{\big(|p+\overline{q_0}|+|p-q_0|\big)^2}
{|p+\overline{q_0}|^2-|p-q_0|^2}.
\end{equation}
By the definition of $p$ in (\ref{def:p-T-23}), we have
$$|p|=|q|,\qquad \textrm{Re}(p)=\textrm{Re}(q).$$
so that
(\ref{213}) leads to
\begin{equation}\label{213-1}
\begin{split}
|h(q)|
&\leq\frac{\big(|q|+|q_0|\big)^2}{{\rm{Re}}(q){\rm{Re}}(q_0)}.
\end{split}
\end{equation}
This implies that
\begin{equation}\label{214}
\begin{split}
|q^{-1}g(q)|& \leq\big|q^{-1}{\rm{Im}}g(q_0)\big|+\big|q^{-1}(g(q)-{\rm{Im}}g(q_0))\big|\\
&=\frac{|{\rm{Im}}g(q_0)|}{|q|}+|h(q)|\frac{{|\rm{Re}}g(q_0)|}{|q|}\\
&\leq\frac{|{\rm{Im}}g(q_0)|}{|q|}+
\frac{\big(|q|+|q_0|\big)^2}{{\rm{Re}}(q){\rm{Re}}(q_0)}\frac{{\rm{Re}}g(q_0)}{|q|}.
\end{split}
\end{equation}
Then for all $q\in\mathcal{S}_{\gamma}$ we have
\begin{equation}\label{215}
\begin{split}
|q^{-1}g(q)|&\leq\frac{|{\rm{Im}}g(q_0)|}{|q|}+
\frac{\big(|q|+|q_0|\big)^2}{\gamma |q|^2}\frac{{\rm{Re}}g(q_0)}{{\rm{Re}}(q_0)}\\
&=\frac{|{\rm{Im}}g(q_0)|}{|q|}+
\frac1{\gamma} \bigg(1+\frac{|q_0|}{|q|}\bigg)^2\frac{{\rm{Re}}g(q_0)}{{\rm{Re}}(q_0)}.
\end{split}
\end{equation}

By the definition of $g$ in (\ref{def:g-c-45}), it is evident that
$$\inf\bigg\{\frac{{\rm{Re}}g(q)}{{\rm{Re}}\,q}:q\in\mathbb H^{+}\bigg\}=0.$$
For any $\epsilon>0$, there thus exists a $q_0\in\mathbb H^{+}$ such that
$$\frac{{\rm{Re}}\; g(q_0)}{{\rm{Re}}(q_0)}\leq\gamma\,\epsilon,$$
and hence by letting $|q|\rightarrow\infty$ in (\ref{215}) we obtain
$$\mathop{\limsup_{|q|\rightarrow\infty}}_{q\in \mathcal{S}_{\gamma}}\big|q^{-1}g(q)\big|\leq \epsilon.$$
Namely,
 \begin{equation}\label{216}
 \mathop{\lim_{|q|\rightarrow\infty}}_{q\in \mathcal{S}_{\gamma}}q^{-1}g(q)=0,
\end{equation}
since  $\epsilon $ is an arbitrary positive number. At the same time, for all $q\in\mathcal{S}_{\gamma}$,
$$\frac{{\rm{Re}}g(q)}{{\rm{Re}}(q)}\leq\frac{|g(q)|}{|q|}\,\frac{|q|}{{\rm{Re}}(q)}
\leq\frac1{\gamma}\big|q^{-1}g(q)\big|,$$
which together with $(\ref{216})$ implies that
$$\mathop{\lim_{|q|\rightarrow\infty}}_{q\in \mathcal{S}_{\gamma}}\frac{{\rm{Re}}g(q)}{{\rm{Re}}(q)}=0.$$
Now the proof of the assertions $(\textrm{i})$ and $(\textrm{ii})$ is complete.

 It remains to prove $(\textrm{iii})$. For any given $\gamma\in(0,1/2)$, it is easy to see that the closed ball $$\overline{\mathbb B}(q,\gamma|q|)\subset\mathcal{S}_{\gamma},$$
whenever $q\in\mathcal{S}_{2\gamma} $. For any $I\in\mathbb S$ and
$q\in\mathcal{S}_{2\gamma}\cap\mathbb C_I$ we have
\begin{equation}\label{Formula:cauchy-integral}
f'(q)-c=\frac1{2\pi I}
\int_{\partial\mathbb B(q,\gamma|q|)\cap\mathbb C_I}\frac{ds}{(s-q)^2}\big(f(s)-cs\big).
\end{equation}
Notice that for any $s\in \partial\mathbb B(q,\gamma|q|)\cap\mathbb C_I$ we have
$|s-q|=\gamma |q|$
so that
$$(1-\gamma)|q|\leq |s|\leq (1+\gamma)|q|.$$
Since $| f(s)-cs\big|=|s|\big|s^{-1} f(s)-c\big|$ and
$$
 \partial\mathbb B(q,\gamma|q|)\cap\mathbb C_I\subseteq \big\{s\in \mathcal{S}_\gamma: |s|\geq (1-\gamma)|q|\big\}
$$
 it follows from (\ref{Formula:cauchy-integral}) that
\begin{equation}\label{217}
\begin{split}
\big|f'(q)-c\big|
\leq \Big(1+\frac1{\gamma}\Big)\mathop{\sup\limits_{|s|\geq (1-\gamma)|q|)}}_{s\in \mathcal{S}_{\gamma}}\big|s^{-1} f(s)-c\big|,
\end{split}
\end{equation}
which approaches to $0$ as $q$ tends to $\infty$ in $\mathcal{S}_{2\gamma}$
due to assertion (ii). Consequently,
$$\mathop{\lim_{|q|\rightarrow\infty}}_{q\in \mathcal{S}_{2\gamma}}f'(q)=c.$$
Now the proof is complete.
\end{proof}
\begin{rem}
There exists a regular function $f:\mathbb H^+\rightarrow\mathbb H^+$ such that
$$c:=\inf\bigg\{\frac{{\rm{Re}}f(q)}{{\rm{Re}}(q)}:q\in\mathbb H^{+}\bigg\}=0.$$
A simple example is the constant function $f(q)=1$. Furthermore, we shall also denote by $f'(\infty)$ the quantity
$$\inf\bigg\{\frac{{\rm{Re}}f(q)}{{\rm{Re}}(q)}:q\in\mathbb H^{+}\bigg\}$$
due to ${\rm{(iii)}}$ in the preceding theorem.
\end{rem}

Some quite interesting consequences of the preceding theorem are in order.
The first one is the following right half-space version of boundary Schwarz lemma.
\begin{thm}
Let $f:\mathbb H^+\rightarrow\mathbb H^+$ be a regular function with a fixed point $q_0\in\mathbb H^+$. Then
$$f'(\infty)\leq1$$
with equality if and only if $f(q)=q$ for all $q\in\mathbb H^+$.
\end{thm}

\begin{proof}
The inequality $f'(\infty)\leq1$ immediately follows from Theorem
$\ref{JC-half space}$.
If $f'(\infty)=1$, then the nonnegative slice harmonic function ${\rm{Re}}\big(f(q)-q\big)$ attains its minimum at interior point $q_0\in\mathbb H^+$. Thus the desired result easily follows from the maximum principle.
\end{proof}

Another consequence of Theorem
$\ref{JC-half space}$ is the following result concerning  the asymptotic behavior at infinity of regular self-mappings of $\mathbb H^+$.

\begin{cor}
Let $f:\mathbb H^+\rightarrow\mathbb H^+$ be a regular function. Then
 there exists a positive number $\beta$, finite or infinite, such that for each $\gamma\in (0,1)$, we have
$$\mathop{\lim_{|q|\rightarrow\infty}}_{q\in \mathcal{S}_{\gamma}}qf(q)=\beta. $$
\end{cor}
\begin{proof}
By assumption, $f$ has no zeros and so does $f^{-\ast}$ $($see Proposition 3.9 in \cite{GSS}$)$. The result immediately follows by applying  the preceding theorem to $f^{-\ast}$.
\end{proof}

The preceding corollary in turn results in the quaternionic version of Burns-Krantz theorem.
\begin{thm}\label{218-1}{\bf(Burns-Krantz)}
Let $f:\mathbb H^+\rightarrow\mathbb H^+$ be a regular function. If there exists a sequence $\{q_n\}_{n\in\mathbb N}\subset\mathcal{S}_{\gamma}$ for some $\gamma\in(0,1)$ converging to $\infty$ as $n$ tends to $\infty$, such that
\begin{equation}\label{Quater-Burns-Krantz}
f(q_n)=q_n+o\Big(\frac1{q_n}\Big),\qquad as\;\; n\rightarrow \infty,
\end{equation}
 then $f(q)=q$ for all $q\in\mathbb H^+$.
\end{thm}
\begin{proof}
By $\rm{(ii)}$ in Theorem \ref{JC-half space},
$$c:=\inf\bigg\{\frac{{\rm{Re}}f(q)}{{\rm{Re}}(q)}:q\in\mathbb H^{+}\bigg\}=1.$$
Therefore, the regular function $g(q)=f(q)-q$ maps $\mathbb H^+$ into $\overline{\mathbb H}^+$, and satisfies that
$$\mathop{\lim_{|q|\rightarrow\infty}}_{q\in \mathcal{S}_{\gamma}}qg(q)=0. $$
Now it follows from the preceding corollary that $g\equiv 0$ and $f(q)=q$ for all $q\in\mathbb H^+$.
\end{proof}

The third consequence of Theorem
 $\ref{JC-half space}$ is the following rigidity theorem.

\begin{cor}\label{218}
Let $f:\mathbb H^+\rightarrow\overline{\mathbb H}^+$ be a regular function. If there exist a $I\in\mathbb S$ and some $\theta\in(-\frac{\pi}2,\frac{\pi}2)$ such that  $$\liminf\limits_{r\rightarrow \infty}r\big|f(re^{I\theta})\big|=0,$$
 then $f\equiv0$.
\end{cor}

Finally, we give the proof of the Burns-Krantz theorem.

\begin{proof}[Proof of Theorem $\ref{B-krantz}$]
Let $\varphi$ be the Cayley transformation from $\mathbb H^{+}$ to $\mathbb B$, i.e.
$$\varphi(q)=(1+q)^{-1}(1-q),\qquad \forall\, q\in \mathbb H^{+}.$$ The result immediately follows by applying the preceding corollary to the regular function $g=f\circ\varphi$.
\end{proof}

\bigskip

\bibliographystyle{amsplain}

\end{document}